\documentclass[11pt]{article}
\usepackage{dsfont}
\usepackage{mathtools}
\usepackage{amsfonts}
\usepackage{amsthm}
\usepackage{amssymb}

\usepackage{mathrsfs}
\usepackage{tflpi-macros-ieee}

\usepackage[usenames]{color}

\title{Transverse feedback linearization with partial information for
  single-input systems\thanks{This work was supported by supported by
    the National Science and Engineering Research Council (NSERC) of
    Canada.}}

\author{Christopher Nielsen\thanks{Department
of Electrical and Computer Engineering, University of Waterloo, Waterloo,
ON, N2L 3G1 Canada. E-mail: {\tt cnielsen@uwaterloo.ca}.}}

\begin{document}

\maketitle

\begin{abstract}
  This paper is motivated by the problem of asymptotically stabilizing
  invariant sets in the state space of control systems by means of
  output feedback. The sets considered are smooth embedded in
  submanifolds and the class of system is nonlinear,
  finite-dimensional, autonomous, deterministic, single-input and
  control-affine. Given an invariant set and a control system with
  fixed output, necessary and sufficient conditions are presented for
  feedback equivalence to a normal form that facilities the design of
  output feedback controllers that stabilize the set using existing
  design techniques.
\end{abstract}

\section{Introduction}
\label{sec:intro}
In this paper the problem of asymptotically stabilizing sets using
output feedback is investigated. Many control objectives can be
accomplished by stabilizing an appropriate invariant set, often a
submanifold, in the state space of a control system. This point of
view is relevant in applications such as output
regulation~\cite{Dav76},~\cite{Fra77},~\cite{IsiByr90},
synchronization~\cite{PogSanNij02}, formation control problems for
multi-agent systems~\cite{ElhMag13},~\cite{KriFraBro09} and path
following~\cite{FreRobShiJoh08},~\cite{NieFulMag10}.  Topological
obstructions for submanifold stabilization using full-state feedback
were characterized in~\cite{Man10}. Output feedback controllers are
necessary whenever the state of the system is not available for
feedback. This is common in applications where, due to economic or
technological reasons, sensors cannot measure a system's entire state.

The most natural approach to stabilizing sets using output feedback is
to find an observable, i.e., available for feedback, function that
yields a well-defined relative degree whose associated zero dynamics
manifold coincides with the set to be stabilized. If such an
observable function exists, then the set stabilization problem becomes
an output stabilization problem which can be solved using classical
and well-understood output feedback control design
techniques~\cite{AndPra09},~\cite{AtaKha99},~\cite{MarPraIsi07},~\cite{TeePra95}. The
main contribution of this paper are necessary and sufficient
conditions for the existence of such an observable function
(Theorem~\ref{thm:main}).



In the case of full-state feedback, i.e., the full information case,
the above approach to set stabilization was studied
in~\cite{NieMag08}. There we sought a coordinate and feedback
transformation locally bringing the control system to a ``normal
form,'' in which the system's dynamics are decomposed into two
cascade-connected subsystems. In our normal form, the driving system
is linear, time-invariant, and controllable.  It models the dynamics
``transversal'' to the target set in the sense that, in transformed
coordinates, the target set corresponds to the origin of this linear
system. We refer to the driving subsystem as the transversal
subsystem. On the other hand, the restriction of the driven system to
the target set represents the ``tangential'' motion of the control
system on the set, and for this reason such restriction is referred to
as the tangential subsystem. The process of bringing the original
control system to the normal form just described is called local
transverse feedback linearization (LTFL). This terminology originated
with the work of Andrzej Banaszuk and John Hauser
in~\cite{BanHau95}.

\subsection{Contributions}
The contributions of the paper are the following. 1) The results
in~\cite{NieMag06},~\cite{NieMag08} are extended to the partial
information case in Theorem~\ref{thm:main}. In that work we assumed
that the full state of the control system is available for feedback
and that, in particular, the local transverse output was permitted to
be a function of the entire state. In this paper we assume that the
only information available for feedback is modeled by a fixed
output. We refer to this as the partial information case. This work is
complementary to the
papers~\cite{AtaKha99},~\cite{MarPraIsi07},~\cite{MarPraIsi10}~\cite{DelliPriscoli2009},~\cite{TeePra95}. In
those papers output feedback controllers are designed for systems in a
given normal form. The main result of this work, motivated by set
stabilization problems, provides necessary and sufficient conditions
under which a system can be brought into the aforementioned normal
form in which the control design techniques can be applied. 2)
Sufficient conditions under which a global version of the problem can
be solved are presented in Section~\ref{sec:global}. 3) In
Section~\ref{sec:disturbances} we show how the results of this paper
can be used to facilitate control design for systems affected by
unmeasured disturbances. A preliminary version of this paper appeared
in~\cite{NieMag12}.

\section{Motivating example}
\label{sec:motivation}
Consider a system
\begin{equation}
\left[
  \begin{array}{c}
\dot x_1 \\  
\dot x_2 \\  
\dot x_3 \\  
\dot x_4 \\
\dot x_5
  \end{array}
\right] = \left[
  \begin{array}{c}
x_4    \\
-x_3-x_2^3\\
x_2\\
0\\
x_1
  \end{array}
\right] + \left[
  \begin{array}{c}
x_1\\0\\0\\1\\x_5    
  \end{array}
\right]u
\label{eq:example1}
\end{equation}
with output
\begin{equation}
y = h(x) = \left[
  \begin{array}{c}
x_4 \\  
x_5
  \end{array}
\right].
\label{eq:ex_out}
\end{equation}
We are interested in locally stabilizing the invariant set
\[
\Gamma^\star = \set{x \in \Real^5 : x_1 = x_4 = x_5 = 0}.
\]
As discussed in Section~\ref{sec:intro}, the most direct approach to
locally stabilizing this set using output feedback is to seek an
observable function that yields, in a \nbhd of a point $x_0 \in
\Gamma^\star$, a well-defined relative degree whose associated zero
dynamics manifold coincides with $\Gamma^\star$. If such a function
exists, then zeroing the function locally solves the set stabilization
problem (if the trajectories of the closed-loop system are
bounded). Furthermore, if the function is observable,
system~\eqref{eq:example1} is feedback equivalent to a system that
fits the framework of well-known and ``standard'' output feedback
control design approaches~\cite{AtaKha99},~\cite{TeePra95} that can
zero the observable function.

Specifically, in this example we seek a function $\lambda: \Real^5 \to
\Real$ that, in a \nbhd of a point $x_0 \in \Gamma^\star$, has the
following properties.
\begin{enumerate}
\item The function $\lambda$ yields a well-defined relative degree at
  $x_0 \in \Gamma^\star$.
\item The zero dynamics manifold of System~\eqref{eq:example1} with
  output $\lambda$ coincides with the target set $\Gamma^\star$ in a
  \nbhd of $x_0$.
\item The function $\lambda$ is observable. In other words, the
  function $\lambda$ can be expressed as a composition $\lambda =
  \tilde{\lambda} \circ h$ of a sufficiently smooth function
  $\tilde{\lambda} : \Real^2 \to \Real$ with the
  output~\eqref{eq:ex_out}.
\end{enumerate}

A natural first attempt to finding a function with the aforementioned
properties is to check if any of the constraints that define
$\Gamma^\star$ satisfy the conditions enumerated above. In this
example neither of the functions $x_1$ and $x_5$ yield a well-defined
relative degree at any point on $\Gamma^\star$. Furthermore, $x_1$ is
not observable. The constraint function $x_4$ does yield a
well-defined relative degree and is observable, however the zero
dynamics manifold associated to the output $x_4$ does not equal
$\Gamma^\star$. Hence making $x_4 \To 0 $ does not ensure that the set
$\Gamma^\star$ is locally attractive. These facts mean that it is not
clear whether or not a function $\lambda$ that satisfies the three
conditions above exists and, therefore, it is not clear whether or not
the above program can be carried out.

The main contribution of this paper is to provide, given a control
system with fixed output and an invariant set, necessary and
sufficient conditions for the existence of a function $\lambda$ that
satisfies the three conditions listed above. In this example the
function
\begin{equation}
\lambda(x) = \tilde{\lambda} \circ h(x) = x_5\e^{-x_4}
\label{eq:lambda_ex}
\end{equation}
meets the above criteria with $\tilde{\lambda}(y) = y_2\e^{-y_1}$. We
now illustrate how this function facilities output feedback
stabilization of $\Gamma^\star$.

Using the observable function~\eqref{eq:lambda_ex} define the
coordinate transformation
\[
\left[
  \begin{array}{c}
\eta_1 \\  
\eta_2 \\  
\xi_1 \\  
\xi_2 \\  
\xi_3
  \end{array}
\right] \coloneqq \left[
  \begin{array}{c}
x_2    \\
x_3\\
x_5\e^{-x_4}\\
x_1\e^{-x_4}\\
x_4\e^{-x_4}
  \end{array}
\right]
\]
which, by the inverse function theorem, is a diffeomorphism of a \nbhd
of any point $x \in \Real^5$.  The system in $(\eta, \xi)$-coordinates
reads
\[
\left[
  \begin{array}{c}
\dot \eta_1 \\  
\dot \eta_2 \\  
\dot \xi_1 \\  
\dot \xi_2 \\  
\dot \xi_3
  \end{array}
\right] = \left[
  \begin{array}{c}
-\eta_2-\eta_1^3    \\
\eta_1\\
\xi_2\\
\xi_3\\
0
  \end{array}
\right] + \left[
  \begin{array}{c}
0\\0\\0\\0\\1    
  \end{array}
\right]\phi(\eta, \xi)u
\]
where $\phi(\eta, \xi)=\left.\left(1 - x_4\right)\e^{-x_4}\right|_{x =
  T^{-1}(\eta, \xi)}$ and where $\xi_1 = \lambda(x)$ is available for
feedback. For this system there are various approaches one can take to
stabilize the $\xi$-subsystem. For example, one can view the
stabilization problem as the study of a system with unknown high
frequency gain for which the techniques in~\cite{Nuss83} along with
the switching strategy in~\cite{ILCOWE91} can be employed to stabilize
$\Gamma^\star$. Alternatively, using the results in~\cite{AtaKha99},
there exists a dynamic feedback that stabilizes $\Gamma^\star$ using
only measurements of $\xi_1$. The ``high-gain'' observer used
in~\cite{AtaKha99} takes the form
\[
\dot{\hat{\xi}} = \left[\begin{array}{ccc}0 & 1 & 0\\ 0 & 0 & 1\\0 &
    0 & 0\end{array}\right]\hat{\xi} + \left[
  \begin{array}{c}
    0\\0\\1    
  \end{array}
\right]\phi_0(\hat{\xi})u + \left[
  \begin{array}{c}
    \frac{\alpha_1}{\varepsilon}\\\frac{\alpha_2}{\varepsilon^2}\\\frac{\alpha_3}{\varepsilon^3}
  \end{array}
\right]\left(\xi_1
      - \hat{\xi}_1\right)
\]
where $\hat{\xi} \in \Real^3$ is an estimate of $\xi$,
$\phi_0(\hat{\xi}) = (1 - \hat{\xi}_3)$ is the nominal, inexact, model
of $\phi(\eta, \xi)$, $\varepsilon > 0$ is a high-gain parameter and
the constants $\alpha_i$ are chosen so that the polynomial $s^3 +
\alpha_1s^2 + \alpha_2s + \alpha_3$ is Hurwitz. The control law is
chosen as
\[
u = \frac{1}{\phi_0(\hat{\xi})}\left(-k_1\hat{\xi}_1 -k_2\hat{\xi}_2 -k_3\hat{\xi}_3 \right)
\]
with $k_i > 0$, $i \in \set{1,2,3}$.

The applicability of the above approach to output feedback control
design depends crucially on the existence of the observable
function~\eqref{eq:lambda_ex}. Therefore, a key challenge in output
feedback stabilization of invariant sets is finding such a function
and, most importantly, in determining whether it exists or not. This
paper completely solves the latter question for single-input systems.

\section{Preliminaries}
This section presents the notation used throughout the
paper. Section~\ref{sec:bundles} contains supporting material needed
to prove the main result. Section~\ref{sec:vectorfields} provides
definitions for the concept of invariance used in this paper and the
Lie derivative and Lie bracket.

\subsection{Notation}
Let $\col{(x_1 \ldots, x_k)} \coloneqq \left[\begin{array}{ccc}x_1 &
    \cdots & x_n\end{array}\right]^\top$ where ${}^\top$ denotes
transpose.  Let $x$ and $y$ be two column vectors, define $\col{(x,y)}
\coloneqq \left[\begin{array}{cc}x^\top &
    y^\top\end{array}\right]^\top$. If $x \in \Real^n$ then $\|x\|$
denotes the Euclidean norm. If $\mathscr{V}$ and $\mathscr{W}$ are
subspaces of the finite-dimensional vector space $\mathscr{X}$, the
notation $\mathscr{V} \oplus \mathscr{W}$ (internal direct sum)
represents the subspace $\mathscr{V} + \mathscr{W}$ when $\mathscr{V}$
and $\mathscr{W}$ are independent.

If $f$ is a scalar-valued function from an open set $U \subseteq
\Real^n$ into $\Real$, and $k$ times continuously differentiable for
at every $x \in U$, then $f$ is of differentiability class $C^k$ on
$U$, denoted $f \in C^k(U)$ or $f \in C^k$ when the domain of $f$ is
clear. If $f$ is $C^k$ for all $k$, then $f$ is $C^\infty$ or
smooth. If $f : U \subseteq \Real^n \to V \subseteq \Real^m$ is a
continuously differentiable map, then for each $x \in U$, the
derivative of $f$ at $x$, denoted $\D f_x$ , is a linear map $\D f_x :
\Real^n \to \Real^m$. Its matrix representation is the Jacobian matrix
of $f$ evaluated at $x$. If $U$ is an open set of $\Real^n$, let
$\Diff{U}$ denote the family of diffeomorphism with domain $U$.

For brevity, the term submanifold is used in place of embedded
submanifold of $\Real^n$. If $M$ is a smooth manifold and $ p \in M$,
we denote by $T_pM$ the tangent space to $M$ at $p$ and by $TM$ the
tangent bundle of $M$. The cotangent space to $M$ at $p$ is denoted by
$T^\star_pM$ and the cotangent bundle is written as $T^\star M$. 
\begin{definition}
\label{def:distribution}
A smooth distribution $D$ on a manifold $M$ is an assignment to each
$p \in M$ of a subspace $D(p) \subseteq T_pM$ which varies smoothly as
a function of $p$. A point $p \in M$ is a regular point of the smooth
distribution $D$ if there exists a \nbhd $U$ containing $p$ for which
$\dim{(D(q))}$ is constant for all $q \in U$. In this case, $D$ is
said to be {nonsingular} on $U$. Similarly, a codistribution $\Omega$
on $M$ assigns at each $p \in M$ a subspace $\Omega(p) \subseteq
T^\star_pM$.
\end{definition}

Given a smooth distribution $D$, we let $\inv{(D)}$ be its involutive
closure (the smallest involutive distribution containing $D$). The
codistribution $\ann{(D)}$ is the annihilator of $D$, i.e., an
assignment to each $p \in M$ of a subspace $\ann{(D)}(p) \subseteq
T_p^\star M$ with the property that if $\sigma \in \ann{(D)}(p)$ and
$\tau \in D(p)$, then $\sigma(\tau) = 0$.

\subsection{Vector bundles}
\label{sec:bundles}
If $D$ is a distribution defined on $\Real^n$ and $N$ is a submanifold
we at times consider objects like $TN + D$ and $TN \cap D$. These
objects are examples of real vector bundles, more precisely,
subbundles of $\left.T\Real^n\right|_N$. They are defined, for each $p
\in N$, by $T_pN+ D(p)$ and $T_p N \cap D(p)$, respectively.  These
subbundles, and the operations on them, can be defined formally using
the framework of vector bundles~\cite{Hir76, Lee02, Spi05}.

\begin{definition}
A $n$-dimensional (real) vector bundle is a map
\[
\pi : E \rightarrow B
\]
of manifolds $E$ and $B$ such that, for any $b \in B$, the inverse
image $\pi^{-1}(b)$ has the structure of the $n$-dimensional vector
space $\Real^n$ having the following property of local triviality: For
each $b \in B$, there exists a \nbhd $U$ of $b$ in $B$ and a
diffeomorphism
\[
h: \pi^{-1}(U) \rightarrow U \times \Real^n
\]
such that for every $b^\prime \in U$ the assignment of $x \in
\pi^{-1}(b^\prime)$ to $h(x) = (b^\prime, \hat{h}(x))$ is an
isomorphism of $\pi^{-1}(b^\prime)$ to $\{b^\prime\} \times
\Real^n$. The manifold $E$ is called the total space, $B$ is
called the base space and the vector space $E_b \coloneqq \pi^{-1}(b)$
is called the fibre over $b$.
\label{def:vecbundle}
\end{definition}

All of the vector bundles encountered in this paper are
finite-dimensional and real. We typically denote a vector bundle
$(\pi, E, B)$ by $E$ alone. Given a bundle $(\pi, E, B)$, for each
$b\in B$ we can replace the fibre $\pi^{-1}(b)$ with different vector
spaces. In this paper we will only consider the simplest case, we
replace each vector space $\pi^{-1}(b)$ with its dual space.

\begin{definition}
Let $\xi = (\pi, E, B)$ be a vector bundle. The dual bundle to $\xi$,
is $\xi^\star = \left(\pi^\star, E^\star, B\right)$ where
\[
E^\star \coloneqq \bigcup_{b \in B} \left(\pi^{-1}(b)\right)^\star,
\]
and $\pi^\star: E^\star \rightarrow B$ is the natural projection
$\pi^\star : \left(\pi^{-1}(p)\right)^\star \mapsto p$.
\end{definition}

When this construction is applied to the tangent bundle $TM$ of a
manifold $M$, the resulting bundle is the cotangent bundle $T^\star M$
of $M$.

\begin{definition}
Let $\eta=(\pi_F, F, B)$ be a subbundle of the smooth vector bundle
$\xi = (\pi, E, B)$. The annihilator $\ann{(\eta)}$ of $\eta$, is the
subbundle of $\xi^\star$ whose fibres are defined at each $b \in B$ by
\[
\begin{aligned}
\ann{\left(F_b\right)} &\coloneqq \{e^\star \in E^\star_b: e^\star(f) =
0, \; \forall \ f \in F_b\}.
\end{aligned}
\]
\label{def:annvb}
\end{definition}

Recall that, if $\mathscr{X}$ is a finite dimensional vector space,
then $\left(\mathscr{X}^\star\right)^\star =:
\mathscr{X}^{\star\star}$ is canonically isomorphic to
$\mathscr{X}$. Using this fact, and applying
Definition~\ref{def:annvb} twice to the vector bundle $\xi$, we obtain
the following.

\begin{proposition}
Let $\xi = (\pi, E, B)$ be a smooth vector bundle over $B$. Then
\[
\ann{(\ann{(\xi)})} = \xi.
\]
\label{prop:annann}
\end{proposition}

Proposition~\ref{prop:annann} implies that, if $\eta^\star$ is a
subbundle of the cotangent bundle $T^\star M$, then
$\ann{\left(\eta^\star\right)}$ is a subbundle of $TM$, the tangent
bundle to $M$. The following results, needed in this paper, can be
found in~\cite{Hir76, Lee02, Spi05}.

\begin{proposition}
Let $\xi = (\pi, E, B)$, $\xi_1 = (\pi_{E_1}, E_1, B)$ and $\xi_2 =
(\pi_{E_2}, E_2, B)$ be vector bundles such that $\xi_2 \subseteq
\xi_1 \subseteq \xi$, then,
\[
\ann{(\xi)} \subseteq \ann{(\xi_1)} \subseteq \ann{(\xi_2)} \subseteq
\xi^\star.
\]
\label{prop:ann12}
\end{proposition}

\begin{proposition}
Let $\xi_1 = (\pi_{E_1}, E_1, B)$, $\xi_2 = (\pi_{E_2}, E_2, B)$ be
subbundles of the smooth vector bundle $\xi = (\pi, E, B)$. If $\xi_1
+ \xi_2 = \left(\pi_{F}, F, B\right)$ is also a subbundle of $\xi$,
then
\[
\ann{\left(\xi_1 + \xi_2\right)} = \ann{(\xi_1)}\cap \ann{(\xi_2)}.
\]
\label{prop:annint}
\end{proposition}

\subsection{Invariant sets, Lie derivatives, Lie brackets}
\label{sec:vectorfields}
Denote the set of all $C^\infty$-vector fields on a smooth manifold
$M$ by ${\mathsf V}(M)$. Given $v \in {\mathsf V}(M)$ and a point $p
\in M$, we denote the maximal integral curve, or flow, generated by
the vector field $v$ through the point $p$ as $\phi^v_t(x)$.

\begin{definition}
  A set $N \subset M$ is said to be invariant under $v \in {\mathsf
    V}(M)$ if
\[
\left(p \in N\right) \Rightarrow (\forall t \geq 0)(\phi^v_t(p)
\in N).
\]
\label{def:invariant}
\end{definition}

The property of invariance in Definition~\ref{def:invariant} is
sometimes called positive or forward invariance because $N$ is
invariant for $t \geq 0$. When $N$ is a closed submanifold invariance
for $t \geq 0$ is equivalent to invariance for $t \in \Real$. If $N$
is an $n$-dimensional submanifold of $M$ expressed as
%
%
$N = \{p \in M: \phi(p) = 0\}$,
%
%
where $\phi(p) = \col(\phi_1(p), \ldots, \phi_{m-n}(p))$ is a smooth
map $M \rightarrow \Real^{m-n}$, and $0$ is a regular value of
$\phi$, then there is a particularly simple criterion for invariance.

\begin{theorem}
Let $\phi : M \rightarrow \Real^{m-n}$ be a smooth map, and $0$
be a regular value of $\phi$. Let $v \in {\mathsf V}(M)$, then, $N =
\phi^{-1}(0)$ is invariant under $v$ if, and only if,
\[
\left(\D\phi_i\right)_p(v(p)) = 0
\]
for all $i \in \{1, \ldots, m-n\}$ and all $p \in \phi^{-1}(0)$.
\label{thm:inv_sub}
\end{theorem}

Geometrically, the theorem asserts that $N$ is invariant under $v$ if
and only if $v$ is tangent to $N$, everywhere on $N$. The same is true
for general closed submanifolds of $M$.





\begin{definition}
If $v \in {\mathsf V}(M)$ and $\lambda\in C^\infty(M)$ then the
derivative of $\lambda$ along $v$ is a function $L_v\lambda: M
\rightarrow \Real$ defined by
\[
L_v\lambda(p) = \lim_{h\rightarrow
0}\frac{1}{h}\left[\lambda(\phi^v_h(p)) - \lambda(p)\right]
\]
and called the Lie or directional derivative of $\lambda$ along $v$ at
$p$. It is an element of $C^\infty(M)$.
\label{def:liederivative}
\end{definition}

\begin{definition}
  If $f, g \in {\mathsf V}(M)$, then the Lie bracket of $f$ and $g$ is
  a vector field $\liebr{f}{g} \in {\mathsf V}(M)$ defined by the
  relation
\[
\left(\forall \lambda \in C^\infty(M)\right) \qquad L_{\left[f,
    g\right]}\lambda = L_f (L_g \lambda) - L_g (L_f \lambda).
\]
\label{def:liebracket}
\end{definition}
Definitions~\ref{def:liederivative} and~\ref{def:liebracket} are implicit
in that they do not directly indicate how to compute, respectively,
the Lie derivative and Lie bracket. 
If $\lambda \in C^\infty(\Real^n)$ and $v \in \mathsf{V}(\Real^n)$
then $L_v\lambda(x)$ is computed as
\[
L_v\lambda(x) = \left(\D\lambda\right)_x(v(x)).
\]
If $f, g \in {\mathsf V}(\Real^n)$, the Lie bracket of $f$ and $g$ is
computed as
\[
\liebr{f}{g}(x) = \D g_x\left(f(x)\right) - \D f_x\left(g(x)\right),
\]
where $\D f_x$, $\D g_x$ are the derivative maps of the vector
functions $f, g: \Real^n \rightarrow \Real^n$. We use the following
standard notation for iterated Lie derivatives and Lie brackets
\[
\begin{aligned}
&L^0_g\lambda \coloneqq \lambda, \; \;
\; L^k_g\lambda \coloneqq L_g(L^{k-1}_g\lambda),\\
&L_gL_f\lambda \coloneqq L_g(L_f\lambda), \\ &ad^0_fg \coloneqq g, \; \; \;
ad^k_fg \coloneqq \left[f, ad^{k-1}_fg\right], \; \; \; k \geq 1.
\end{aligned}
\]
%
%

\section{Problem formulation}
\label{sec:problem}
Consider a control system modeled by equations of the form
\begin{equation}
\dot{x} = f(x) + g(x)u.
\label{eq:system} 
\end{equation}
Here $x \in \Real^n$ is the state and $u \in \Real$ is the control
input. The vector fields $f$ and $g : \Real^n \rightarrow T\Real^n$
are smooth ($C^\infty$). Suppose that the state $x$ is not available
for feedback but, rather, the only available information is given by a
smooth vector output
\begin{equation}
\begin{aligned}
y = h(x), \qquad h: \Real^n \rightarrow \Real^p.
\end{aligned}
\label{eq:output}
\end{equation}
We assume that the component functions $\col{(h_1(x), \ldots,
  h_p(x))}$ of the output $h(x)$ are linearly independent,
i.e., we assume that $\D h_x$ has rank $p$ for all $x \in \Real^n$.
Define the following distributions associated with control
system~\eqref{eq:system}
\begin{equation}
\begin{aligned}
\Gg_i \coloneqq &\Sp\{ad^j_fg : 0 \leq j \leq i\}.
\end{aligned}
\label{eq:G}
\end{equation}
To the output~\eqref{eq:output} we associate the nonsingular,
involutive, $(n-p)$-dimensional distribution
\begin{equation}
\label{eq:W}
\W \coloneqq \ann{\left(\Sp{\left\{\D h_1, \ldots ,
    \D h_p\right\}}\right)}.
\end{equation}
Suppose that we are given a submanifold $\Gamma^\star \subset \Real^n$
of dimension $0 < n^\star < n$ which is either invariant under the
vector field $f(x)$ in~\eqref{eq:system} or controlled invariant,
i.e., it can be made invariant by appropriate choice of smooth
feedback.

\begin{definition}
\label{def:cis}
A closed connected submanifold $N \subset \Real^n$ is called
controlled invariant for~\eqref{eq:system} if there exists a smooth
feedback $\overline{u}: N \rightarrow \Real$ making $N$ an invariant
set for the closed-loop system.
\end{definition}

In this paper, as in~\cite{NieMag08}, we treat the controlled
invariant set $\Gamma^\star$ as given data. Often, however, one is
given a set $\Gamma \subset \Real^n$, perhaps defined by virtual
constraints or design goals, and then one must pare away pieces of
$\Gamma$ until all that remains is the maximal controlled invariant
submanifold $\Gamma^\star$ contained in $\Gamma$. We now state the
problem considered in this paper.

\begin{problem*}{Local Transverse Feedback Linearization with Partial
    Information (LTFLPI) Problem }
  Given a smooth single-input system~\eqref{eq:system} with smooth
  output~\eqref{eq:output}, a closed, connected, embedded,
  $n^\star$-dimensional controlled invariant submanifold $\Gamma^\star
  \subset \Real^n$ and a point $x_0 \in \Gamma^\star$, find, if
  possible, a diffeomorphism $\Xi \in \Diff{U}$
\begin{equation}
\begin{aligned}
  \Xi : U &\to \Xi(U) \subset (\Gamma^\star \cap U) \times \Real^{n -
    n^\star}\\
  x &\mapsto (\eta, \xi)
\end{aligned}
\end{equation}
where $U$ is a \nbhd of $x_0$, such that
\begin{itemize}
\item[(i)] The restriction of $\Xi$ to $\Gamma^\star \cap U$ is
\[
\left.\Xi\right|_{\Gamma^\star \cap U} : x \mapsto (\eta, 0). 
\]
\item[(ii)] The dynamics of system~\eqref{eq:system} in $(\eta,
  \xi)$-coordinates reads
\begin{equation}
\begin{aligned}
\dot{\eta} &= f_0(\eta, \xi)\\
\dot{\xi} &= A\xi + b(a_1(\eta, \xi) + a_2(\eta, \xi)u),
\end{aligned}
\label{eq:partial_info}
\end{equation}
where the pair $(A,b)$ is in Brunovsk\'{y} normal form (one chain of
integrators) and $a_2(\eta, \xi) \neq 0$ in $\Xi(U)$.
\item[(iii)] The first component of $\xi$, denoted $\xi_1$ is
  observable, i.e., there exists a function $\tilde{\lambda} : h(U)
  \subseteq \Real^p \to \Real$ such that
\[
\xi_1(x) = \tilde{\lambda} \circ h(x).
\] 
\end{itemize}
\end{problem*}
As illustrated in Section~\ref{sec:motivation}, solving LTFLPI is
relevant for stabilizing the set $\Gamma^\star$ using output feedback.
To understand this claim, suppose that LFTLPI is solvable at $x_0 \in
\Gamma^\star$. Let $\lambda(x) \coloneqq \tilde{\lambda} \circ
h(x)$. Using this function we partially define the diffeomorphism
$\Xi(x)$ by letting $\xi \coloneqq \col{(\lambda(x), L_f\lambda(x), }$
$\ldots, L_f^{n - n^\star - 1}\lambda(x))$. Choose $n - n^\star$
additional independent functions $\eta_i \coloneqq \phi_i(x)$, $i \in
\{1, \ldots, n - n^\star\}$, to complete the coordinate transformation
$\Xi: U \rightarrow \Real^{n^\star} \times \Real^{n - n^\star}$, $x
\mapsto (\eta, \xi)$. In the single-input case, since
$\ann{\left(\Sp{\set{g}}\right)}$ is spanned by exact differentials,
the functions $\phi_i(x)$ can always be chosen (see~\cite{Isi95}) so
that their time derivative along the control system do not depend on
$u$, i.e., so that for all $x \in U$ and all $i \in \{1, \ldots,
n^\star\}$, $L_g\phi_i(x) = 0$. After applying this coordinate
transformation, in $(\eta, \xi)$-coordinates the system is modeled by
equation~\eqref{eq:partial_info}.

If the entire state $x$ is available for feedback, as in the full
information case, then the regular feedback transformation $u =
-\frac{a_1(\eta, \xi)}{a_2(\eta, \xi)} + \frac{v}{a_2(\eta, \xi)}$
yields a system of the form
\begin{equation}
\begin{aligned}
\dot{\eta} &= f_0(\eta, \xi)\\
\dot{\xi} &= A\xi + bv
\end{aligned}
\label{eq:full_info}
\end{equation}
and we say that system~\eqref{eq:system} has been locally
transversely feedback linearized with respect to the set
$\Gamma^\star$. In this case, stabilizing the subspace $\xi = 0$ in
$(\eta, \xi)$-coordinates corresponds to stabilizing the set
$\Gamma^\star \cap U$ in original coordinates (if the trajectories of
the closed-loop system are bounded). For this reason we call the
$\xi$-subsystem of~\eqref{eq:partial_info} the transverse dynamics
of~\eqref{eq:system} with respect to $\Gamma^\star$. Stabilizing
$\xi=0$ can be achieved easily using the auxiliary control input $v$
since the pair $(A,b)$ is controllable. On the target set, the system
dynamics are governed by the ordinary differential equation
\begin{equation}
\dot{\eta} = f_0(\eta, 0).
\label{eq:tang}
\end{equation}
For this reason the dynamics~\eqref{eq:tang} are called the tangential
dynamics of~\eqref{eq:system} with respect to $\Gamma^\star$.

In the partial information case the state $x$ is not available for
feedback, the only available information is given by the output
function~\eqref{eq:output}. In this case $(\eta, \xi)$ is not
available for feedback, the feedback transformation above cannot be
implemented, and it may be impossible to stabilize the $\xi$
subsystem. In the partial information case the $\xi$-subsystem before
feedback transformation is
\begin{equation}
\label{eq:xi_subsys}
\begin{aligned}
\dot \xi_1 &= \xi_2\\
&\cdots\\
\dot \xi_{n-n^\star -1} &= \xi_{n-n^\star}\\
\dot \xi_{n-n^\star} &= a_1(\eta, \xi) + a_2(\eta, \xi)u.
\end{aligned}
\end{equation}
Since $\xi_1 = \lambda(x) = \tilde{\lambda}(h(x))$, it is available
for feedback. For system~\eqref{eq:partial_info} with $\xi_1$ measured
and $a_2(\eta, \xi)$ sign-definite, the results
in~\cite{AtaKha99},~\cite{MarPraIsi07},~\cite{MarPraIsi10}~\cite{DelliPriscoli2009},~\cite{TeePra95}
assert the existence of a dynamic feedback
\[
\begin{aligned}
&\dot \zeta  = \varphi(\zeta, \xi_1)\\
&u = \varrho(\zeta, \xi_1)
\end{aligned}
\]
capable of stabilizing the origin of~\eqref{eq:xi_subsys}. The
selection of an appropriate output feedback design framework depends
crucially on the properties of tangential system~\eqref{eq:tang}. In
the simplest case, the results of~\cite{AtaKha99} can be used whenever
$\phi(\eta, \xi, u) \coloneqq a_1(\eta, \xi) + a_2(\eta, \xi)u$ and
$f_0(\eta, \xi)$ are locally Lipschitz with $\phi(0,0,0)=0$,
$f_0(0,0)=0$ and the tangential dynamics~\eqref{eq:tang} are minimum
phase. Alternatively, if practical stability is sought then the
results of~\cite[Section 6]{TeePra95} can be used, again provided the
tangential subsystem is minimum phase.

In the cases when the tangential dynamics do not necessarily converge
to zero but remain otherwise bounded, the results
in~\cite{MarPraIsi07} are relevant. In~\cite{MarPraIsi07} a
weak-minimum phase assumption is made on the tangential subsystem and
the function $a_1(\eta, \xi)$ is not necessarily known. Furthermore,
they require that the tangential dynamics $f_0(\eta, \xi)$ have the
form $f_0(\eta, \xi_1)$. Sufficient conditions for this additional
property to hold are given in
Corollary~\ref{prop:global_normal_form}. Finally, in cases where the
tangential system has the form $f_0(\eta, \xi_1)$ and $\DER{a_2(\eta,
  \xi)}{\xi_i} \equiv 0$, $i \in \set{2, \ldots, n-n^\star}$ and
$\DER{a_2(\eta, \xi)}{\eta} \equiv 0$, the results in~\cite{ByrIsi91},
see also~\cite{AndPra09}, are applicable.

Once a system is expressed in the normal form~\eqref{eq:partial_info}
there are many other output stabilization techniques one can
consider. The survey~\cite{AndPra09} gives an excellent overview of
the available techniques while also classifying them as direct or
indirect approaches. Conceptually, the direct design approach is
preferable because the state feedback stabilizing control law is known
and therefore the estimation scheme can focus on estimating the
control signal directly. On the other hand, the indirect approach is
far more common in the research literature and in particular, the
approach used in Section~\ref{sec:motivation} is an example of
``domination via a dominant model''~\cite{AndPra09}. Motivated by
these observations, we seek conditions guaranteeing the existence of
an observable transverse output function.
\section{Main result}
\label{sec:main}

The next result, an obvious consequence of~\cite[Theorem
4.1]{NieMag06} or~\cite[Theorem 3.1]{NieMag08}, shows that LFTLPI is
solvable if and only if there exists a ``virtual output'' function
yielding a well-defined relative degree.
%
%
%
\begin{theorem} LTFLPI is solvable at $x_0\in\Gamma^\star$ if and only
  if there exists a smooth $\Real$-valued function $\tilde{\lambda}$,
  defined on a \nbhd of $h(x_0)$ in $\Real^p$ satisfying
\begin{itemize}
\item[(a)] for some \nbhd $U$ of $x_0 \in \Real^n$, $\Gamma^\star \cap
  U \subseteq \{x\in U: \tilde{\lambda} \circ h (x)=0\}$, and
\item[(b)] the system
\begin{equation}
\begin{aligned}
&\dot x=f(x) + g(x) u  \\
&y^\prime = \lambda(x) = \tilde{\lambda} \circ h (x)
\end{aligned}
\label{eq:virtual_output}
\end{equation}
has relative degree $n - n^\star$ at $x_0$.
\end{itemize}
Moreover, if LTFLPI is solvable, then there exists a \nbhd $V
\subseteq U$ of $x_0$ such that, on $V$, a connected component
$\mathcal{Z}^\star$ of the zero dynamics manifold
of~\eqref{eq:virtual_output} coincides with $\Gamma^\star$ :
$\mathcal{Z}^\star \cap V = \Gamma^\star \cap V$.
\label{thm:rel_iff}
\end{theorem}

The proof of Theorem~\ref{thm:rel_iff} is omitted because it is almost
identical to the proof of~\cite[Theorem 4.1]{NieMag06}.

\begin{definition}
  Let $\tilde{\lambda}$ be a smooth $\Real$-valued functions
  satisfying the conditions of Theorem~\ref{thm:rel_iff}. The map
  $\lambda(x) \coloneqq \tilde{\lambda} \circ h(x)$ is called a local
  observable transverse output of~\eqref{eq:system},~\eqref{eq:output}
  with respect to $\Gamma^\star$.
\end{definition}

The main result of this paper, presented next, gives necessary and
sufficient conditions for the existence of an observable transverse
output.

\begin{theorem}
\label{thm:main}
Suppose that $\inv{\left(\Gg_{n-n^\star -2}+ \W\right)}$ is regular at
$x_0 \in \Gamma^\star$. Then LTFLPI is solvable at $x_0$ for
system~\eqref{eq:system} if and only if
\begin{itemize}
\item[(a)] $T_{x_0}\gstar \oplus \Gg_{n-n^\star -1}(x_0) = T_{x_0}\Real^n$
\item[(b)] there exists an open \nbhd $U$ of $x_0$ in $\Real^n$ such
  that, $\left(\forall x \in \gstar \cap U\right)$,
\[
\dim{\left(T_x\gstar \oplus \Gg_{n-n^\star -2}(x)\right)} =
\dim{\left(T_x\gstar \oplus \inv{\left(\Gg_{n-n^\star -2}+
    \W\right)}(x)\right)}.
\]
\end{itemize}
\end{theorem}
\begin{proof}
  Suppose that LTFLPI is solvable at $x_0 \in \Gamma^\star$. Condition
  (a) is coordinate and feedback invariant so it suffices to show that
  it holds for system~\eqref{eq:full_info}. Let $V \coloneqq
  \Xi{(\Gamma^\star \cap U)}$. By the properties of the normal
  form~\eqref{eq:partial_info}, $\Xi(x_0) = \col{(p_0, 0)}$. Hence
  in $(\eta,\xi)$-coordinates

\[
T_{p_0}V + \Gg_{{n} - n^\star - 1}(\col{(p_0, 0)})= \image{\left(\begin{bmatrix}
I_{n^\star} & \star & \star & \ldots & \star \\ 0_{n-n^\star\times
n^\star}& b & Ab & \ldots & A^{n-n^\star-1}b
\end{bmatrix}\right).}
\]
Since $(A,b)$ is a controllable pair it immediately follows that $T_pV
+ \Gg_{{n} - n^\star - 1}(\col{(p, 0)}) = T_p\Real^n$.
Furthermore, since~\eqref{eq:system} is a single-input system and by
the definition~\eqref{eq:G} of $\Gg_i$, $\dim{(\Gg_{{n} - n^\star -
    1})} \leq n - n^\star$, and therefore the subspaces
$T_{x_0}\Gamma^\star$ and $\Gg_{{n} - n^\star - 1}(x_0)$ are
independent which proves that condition $(a)$ is necessary.

We are left to show that condition $(b)$ is necessary. Since LTFLPI
is solvable and $\xi_1(x) = \lambda(x) = \tilde{\lambda} \circ h(x)$,
we have
%
%
$\Gamma^\star \cap U = \{x \in \Real^{{n}}: \xi(x) = 0\} \subseteq \{x
\in \Real^{{n}}: {\lambda}(x) = 0\}$
%
%
so that, for all $x \in \Gamma^\star \cap U$ and for any $v \in
T_{x}\Gamma^\star$,
%
%
$L_v{\lambda}(x) = 0$.
%
%
This implies that $\D\lambda \in \ann{(T\Gamma^\star)}$. Furthermore,
since ${\lambda}(x)$ yields a well-defined relative degree of
$n-n^\star$ at $x_0$, for any $x$ in an open \nbhd of $x_0$, without
loss of generality $U$,
%
%
$L_g{\lambda}(x) = L_{ad_fg}{\lambda}(x) = \cdots =
L_{ad^{n-n^\star-2}_fg}{\lambda}(x) = 0$
%
%
and
%
%
$L_{ad^{n-n^\star - 1}_fg}{\lambda}(x) \neq 0$.
%
%
This means that, in a \nbhd of $x_0$, without loss of generality
$U$,
\[
\D{\lambda} \in \ann{\left(\Gg_{{n} - n^\star - 2}\right)}, \; \; \;
      \D{\lambda} \not\in \ann{\left(\Gg_{{n} - n^\star - 1}\right)}.
\]
By the chain rule, $\D\lambda_x = \D\tilde{\lambda}_{h(x)} \circ \D
h_x,$
%
%
%
%
so that, for any vector field $w \in \W$, and all $x \in U$,
\[
\D\lambda_x(w(x)) = \D\tilde{\lambda}_{h(x)} \circ \D h_x(w(x)) = 0.
\]
In other words, $\D\lambda \in \ann{(\W)} = \Sp\{\D h_1, \ldots, \D
h_p\}$.  This shows that, in $U$,
\[
\begin{aligned}
&\D\lambda \in \ann{\left(\Gg_{{n} - n^\star - 2}\right)} \cap \ann{\left(\W\right)}\\
&\Rightarrow \; \Gg_{{n} - n^\star - 2}+ \W \subseteq \ann{(\D\lambda)}.
\end{aligned}
\]
The distribution $\ann{(\D\lambda)}$ is involutive since its
annihilator is spanned by smooth, exact one-forms. Therefore
%
%
$\inv{(\Gg_{{n} - n^\star - 2}+ \W)} \subseteq \ann{(\D\lambda)}$
%
%
and by Proposition~\ref{prop:ann12},
\[
\begin{aligned}
&\D\lambda \in \ann{\left(\inv{(\Gg_{{n} - n^\star - 2}+ \W)}\right)}.
\end{aligned}
\]
This shows that, on $\Gamma^\star \cap U$, $d\lambda \in
\ann{(T\Gamma^\star)} \cap \ann{\left(\inv{(\Gg_{n - n^\star -
2}+\W)}\right)}$. Thus, by Proposition~\ref{prop:annint},
\[
\begin{aligned}
\D\lambda &\in \ann{\left(T\Gamma^\star + 
\inv{(\Gg_{n - n^\star - 2}+ \W)}\right)}
\end{aligned}
\]
which implies that on $\Gamma^\star \cap U$,
\begin{equation}
\dim{\left(\ann{\left(T\Gamma^\star + \inv{\Gg_{n - n^\star - 2}+
W}\right)}\right)} \geq 1.
\label{eq:dim_ann}
\end{equation}
Therefore, by~\eqref{eq:dim_ann} and Proposition~\ref{prop:annann}, at
any point on $x \in \Gamma^\star \cap U$
\[
\dim{\left(T_{x}\Gamma^\star + \inv{(\Gg_{n - n^\star - 2}+
\W)}(x)\right)} < {n}.
\]
We have already shown that condition $(a)$ is necessary and therefore
on $\Gamma^\star \cap U$,
\begin{equation}
\dim{\left( T_x\Gamma^\star \oplus \Gg_{n -n^\star - 2}(x)\right)} = {n} -1.
\label{eq:dim_TVG}
\end{equation}
Therefore, by~\eqref{eq:dim_ann} and~\eqref{eq:dim_TVG}, we have that
for any point in $\Gamma^\star \cap U$
\begin{equation*}
\begin{aligned}
{n} -1 &= \dim{\left( T_{x}\Gamma^\star + \Gg_{n-n^\star -
    2}\right)(x)}\\ &\leq \dim{\left(T_{x}\Gamma^\star +
  \Gg_{n-n^\star - 2}(x) + \W(x)\right)}\\ &\leq \dim{\left(
  T_{x}\Gamma^\star + \inv{\left(\Gg_{n -n^\star - 2}+ \W\right)}(x)
  \right)} < n.
\end{aligned}
\end{equation*}
which proves the necessity of condition $(b)$.

We now turn to the proof of sufficiency. Conditions $(a)$ and $(b)$,
and the regularity of $\inv{(\Gg_{n-n^\star-2} + \W)}$ at $x_0$ imply
that $T\Gamma^\star \cap \inv{\left(\Gg_{n-n^\star-2}+\W\right)}$ is a
smooth nonsingular distribution near $x_0$. Using an argument
identical to that in the proof of~\cite[Lemma 4.5]{NieMag08}, it can
be shown that, by taking $U$ sufficiently small, there exists a smooth
nonsingular distribution $\Gg^\tang \subset
\inv{(\Gg_{n-n^\star-2}+\W)}$ on $U$ enjoying the two properties
\[
\begin{aligned}
  \left(\forall x \in U\right) \;
  &\inv{\left(\Gg_{n-n^\star-2}+W\right)}(x) = \Gg^\tang(x) \oplus
  \Gg_{n -
    n^\star-2}(x),\\
  \left(\forall x \in \Gamma^\star \cap U\right) \;
  &\left.\Gg^\tang\right|_{\Gamma^\star \cap U}(x) = T_x\Gamma^\star
  \cap \inv{\left(\Gg_{n-n^\star-2}+\W\right)}(x).
\end{aligned}
\]
Let $V \coloneqq \Gamma^\star \cap U$ and let $w_1, \ldots, w_\mu$ be
a set of local generators for $\Gg^\tang$ on $U$. Similarly, there
exist $n^\star - \mu$ vector fields $\{v_1, \ldots, v_{n^\star -
  \mu}\}$, $v_i : V \To TV$, such that, for all $x \in V$,
\[
T_xV = \left.\Gg^\tang\right|_V(x) \oplus \Sp\{v_1, \ldots, v_{n^\star - \mu}\}(x).
\]
Now we have a collection of $n$ linearly independent vector fields,
\begin{equation}
\underbrace{\underbrace{\overbrace{v_1, \ldots, v_{n^\star - \mu}}^{TV
    / \left.\Gg^\tang\right|_V}; \overbrace{w_1, \ldots, w_{\mu}}^{\Gg^\tang}}_{TV}; \overbrace{g, \ldots,
ad^{n-n^\star-2}_fg}^{\Gg_{n - n^\star -2}},
    ad_f^{n-n^\star-1}g}_{T\Real^n}
\label{eq:array_info}
\end{equation}
which we use to generate a coordinate transformation analogous to that
used in Theorem~\cite[Theorem 3.2]{NieMag08}. We work our way from
left to right in the list~\eqref{eq:array_info} starting with the
group of vector fields spanning $TV/\left.\Gg^\tang\right|_V$. Define
the map $S_{\varnothing} \coloneqq \left(s_1, \ldots, s_{n^\star -
  \mu}\right) \mapsto \Phi^\varnothing_{S_\varnothing}(x_0)$ as
\[
\Phi^{\varnothing}_{S_\varnothing}(x_0) = \phi^{v_{n^\star -
  \mu}}_{s_{n^\star - \mu}} \circ \cdots \circ \phi^{v_1}_{s_1}(x_0).
\]
Next define $S^\parallel \coloneqq \left(s^\tang_{1}, \ldots,
  s^\tang_{\mu}\right) \mapsto \Phi^{\tang}_{S^\tang}(x)$, as
\[
\Phi^\tang_{S^\tang}(x) \coloneqq \phi^{w_{\mu}}_{s^\tang_{\mu}} \circ
  \cdots \circ \phi^{w_1}_{s^\tang_{1}}(x),
\]
and the map $S^\tran \coloneqq \left(s^\tran_{0}, \ldots,
s^\tran_{n-n^\star - 2}\right) \mapsto \Phi^\tran_{S^\tran}(x)$, as
%
%
\[
\Phi^\tran_{S^\tran}(x) \coloneqq \phi^{g}_{s^\tran_{0}} \circ \cdots
  \circ \phi^{ad^{n-n^\star -2}_{f}g}_{s^\tran_{n - n^\star -2}}(x).
\]
Finally, let $s \coloneqq (S_\varnothing, s^\tran_{n - n^\star -1},
S^\tran, S^\tang) \mapsto \Phi_s(x_0)$, with domain a \nbhd $U$ of
$s=0$, be defined as
%
%
\begin{equation}
\begin{aligned}
\Phi_s(x_0) \coloneqq \Phi^{\tang}_{S^\tang} \circ \Phi^{\tran}_{S^\tran}
\circ \phi^{ad^{n-n^\star -1}_{f}g}_{s^\tran_{ n - n^\star - 1}}\circ
\Phi^{\varnothing}_{S_{\varnothing}}(x_0).
\end{aligned}
\label{eq:flows}
\end{equation}
Since the vector fields in the list~\eqref{eq:array_info} are linearly
independent near $x_0$, it follows from the inverse function theorem
that there exists a \nbhd $U$ of $s=0$ such that~\eqref{eq:flows} is a
diffeomorphism onto its image. Let
\begin{equation}
\lambda(x) = s^\tran_{n - n^\star - 1}(x).
\label{eq:exoutput}
\end{equation}
We will not show that~\eqref{eq:exoutput} yields a well-defined
relative degree of ${n}-\ns$ at $x_0$ because the arguments are
similar to the proof~\cite[Theorem 3.2]{NieMag08}. Instead, we focus
on showing that there exists a function $\tilde{\lambda}$ such that
$\lambda = \tilde{\lambda}(h(x))$. The function~\eqref{eq:exoutput}
yields relative degree $n -n ^\star$ on near $x_0$, so in particular
\[
L_g{\lambda}(x) = L_{ad_fg}{\lambda}(x) = \cdots =
L_{ad^{n-n^\star-2}_fg}{\lambda}(x) = 0.
\]
Furthermore, since $\Gamma^\star \subset \lambda^{-1}(0)$, these facts imply that,
\[
\begin{aligned}
\D\lambda \in \ann{(T\Gamma^\star)} \cap \ann{({\Gg}_{n - n^\star
-2})} &= \ann{(T\Gamma^\star + {\Gg}_{n - n^\star -2})}\\ &=
\ann{(T\Gamma^\star + \inv{(\Gg_{{n} - n^\star - 2} + W)})}
\\&\subseteq \ann{(T\Gamma^\star + \Gg_{{n} - n^\star - 2} + W)}\\
&= \ann{(T\Gamma^\star)}\cap \ann{(\Gg_{{n} - n^\star - 2} + W)}.
\end{aligned}
\]
Therefore,
\[
\begin{aligned}
\D\lambda &\in \ann{(\Gg_{{n} - n^\star - 2})} \cap \ann{(W)}\\&=
\ann{(\Gg_{{n} - n^\star - 2})} \cap \Sp\{\D h_1, \ldots,
\D h_p\}
\end{aligned}
\]
so
\[
\D\lambda = \sum^{p}_{i=1}\sigma_i(x)\D h_i(x)
\]
which implies that $\lambda = \tilde{\lambda}(h(x))$ and $\sigma_i(x)
= \left.\DER{\tilde{\lambda}}{y_i}\right|_{y = h(x)}$.
\end{proof}
\begin{remark}
  When the state $x$ is available for feedback, i.e., $h(x) = x$, $\W
  = \set{0}$ and $\Gamma^\star = \set{x_0}$ is an equilibrium point of
  the open-loop system, the conditions of Theorem~\ref{thm:main}
  coincide with the necessary and sufficient conditions for solving
  the state-space exact feedback linearization problem~\cite[Theorem
  4.2.3]{Isi95}.
\end{remark}
\begin{remark}
  The direct generalization of Theorem~\ref{thm:main} to multi-input,
  multi-output systems gives sufficient, not necessary, conditions
  under which the MIMO version of LTFLPI is solvable. The MIMO proof
  of necessity is an open problem and the subject of future research.
\end{remark}
\subsection*{Example}
We now return to the motivating example of
Section~\ref{sec:motivation} to illustrate the application of
Theorem~\ref{thm:main}. In Section~\ref{sec:motivation} we had a
system of the form~\eqref{eq:system},~\eqref{eq:output} with $f(x) =
\col{\left(x_4, -x_3-x_2^3, x_2, 0, x_1\right)}$, $g(x) =
\col{\left(x_1, 0, 0, 1, x_5\right)}$, $h(x) = \col{\left(x_4,
    x_5\right)}$. The target set is given by $\Gamma^\star = \set{x
  \in \Real^4 : x_1 = x_4 = x_5 = 0}$.  We now use
Theorem~\ref{thm:main} to justify the discussion from
Section~\ref{sec:motivation} and show that LTFLPI is solvable, for the
given system and set, in a \nbhd of the origin. First note that
$n^\star = 1$ and that
$(\forall x \in \Gamma^\star) \; \; T_x\Gamma^\star = \Sp{\set{v_1,
    v_2}}(x) = \Sp\{e_2, e_3\}$
where $e_2$ and $e_3$ are the second and third natural basis vectors
for $\Real^5$.
Furthermore
\[
\Gg_{2}(x) = \Sp{\set{g, ad_fg, ad^2_fg}}(x) = \Sp{\set{
\left[
   \begin{array}{c}
 x_1\\0\\0\\1\\x_5    
   \end{array}
 \right], \; 
\left[
   \begin{array}{c}
 x_4 - 1\\0\\0\\0\\0    
   \end{array}
 \right], \;
 \left[
 \begin{array}{c}
 0\\0\\0\\0\\1 - x_4    
   \end{array}
 \right]}}.
\]
Checking conditions of Theorem~\ref{thm:main} we have that
$T_0\Gamma^\star + \Gg_{2}(0) = T_0\Real^5 \simeq \Real^5$
so condition (a) holds at $x_0 = 0$. In order to check condition (b)
of Theorem~\ref{thm:main} we first write $\W = \ann{\left(\Sp{\left\{\D h_1, \D h_2\right\}}\right)} =
\Sp{\set{w_1, w_2, w_3}}(x) = \Sp\{e_2, e_3, e_1\}$.
The distribution $\Gg_1 + \W$ is regular in a \nbhd of the origin, it
has dimension four is everywhere. Calculating the Lie brackets between
the vector fields $g$, $ad_fg$, $w_1$, $w_2$ shows that, for all $x
\in \Real^5$, $\Gg_1 + \W(x) = \inv{\left(\Gg_1 + \W\right)}(x)$. This
makes verifying that condition (b) of Theorem~\ref{thm:main} holds
easy to check. We conclude that there exists an observable function
such that Theorem~\ref{thm:rel_iff} holds and hence that LTFLPI is
solvable.

In order to actually find the observable output, in this simple case,
one can follow the semi-constructive procedure of
Theorem~\ref{thm:main}. Construct the maps from the proof of
Theorem~\ref{thm:main} noting that, in a \nbhd of $x_0$,
$T\Gamma^\star/ \Gg^\tang = \{0\}$. We have
\[
\Phi^{\tang}_{S^\tang}(x_0) = \phi^{w_2}_{s^\tang_2} \circ
\phi^{w_1}_{s^\tang_1}(x_0) = \col{\left(x_0, s^\tang_1
    +x_0, s^\tang_2 + x_0, x_0, x_0 \right)}
\]
and $\Phi^{\tran}_{S_\tran}(x) = \phi^g_{s^\tran_0} \circ
\phi^{ad_fg}_{s^\tran_1}(x)$. Hence the overall map~\eqref{eq:flows}
is given by $\Phi_s(x_0) =\Phi^\tang_{S^\tang} \circ
\Phi^{\tran}_{S_\tran} \circ \phi^{ad^2_fg}_{s^\tran_2}(x_0)$ where
\[
\begin{aligned}
s^\tran_0 &\mapsto \phi^{g}_{s^\tran_0}(x) =
\col{\left(\e^{s^\tran_0}x_1, x_2, x_3, s^\tran_0 + x_4, \e^{s^\tran_0}x_5\right)}\\
s^\tran_1 &\mapsto \phi^{ad_fg}_{s^\tran_1}(x) = \col{\left((x_4
    -1)s^\tran_{1} + x_1, x_2, x_3, x_4, x_5\right)}\\
s^\tran_2 &\mapsto \phi^{ad^2_fg}_{s^\tran_2}(x) = \col{\left(x_1,
    x_2, x_3, x_4, (1 - x_4)s^\tran_2 + x_5\right)}.
\end{aligned}
\]
and therefore, with $x_0 = 0$,
\[
\Phi_s(x_0) =\Phi^\tang_{S^\tang} \circ \Phi^{\tran}_{S_\tran}
\circ \phi^{ad^2_fg}_{s^\tran_2}(x_0) =
\col{\left(-s_1^\tran\e^{s^\tran_0}, s^\tang_1, s^\tang_2, s^\tran_0, s^\tran_2\e^{s^\tran_0}\right)}.
\]
Inverting this map yields
\[
\col{\left(s^\tang_1(x), s^\tang_2(x),
    s_0^\tran(x), s^\tran_1(x), s^\tran_2(x)\right)} =
\col{\left(x_2,x_3, x_4, -x_1\e^{-x_4}, x_5\e^{-x_4}\right)}.
\]
The observable transverse output is the function $\lambda(x) =
s^\tran_2(x) = x_5\e^{-x_4}$. This function allows one to follow the
design procedure in Section~\ref{sec:motivation}.

\section{Extensions and applications}
\label{sec:extensions}
In this section we present extensions and applications of
Theorem~\ref{thm:main}. In Section~\ref{sec:global} we pose the global
transverse feedback linearization problem with partial information in
which, roughly speaking, one seeks a single coordinate transformation
such that~\eqref{eq:system} is equivalent to the normal
form~\eqref{eq:partial_info} in a \nbhd of the entire set
$\Gamma^\star$. We present sufficient conditions for the global
problem to be solvable by restricting the geometry of the target
set. In Section~\ref{sec:disturbances} we illustrate how these results
can be applied to systems affected by disturbances that cannot be
measured.

\subsection{Global transverse feedback linearization with partial information}
\label{sec:global}

The results of Section~\ref{sec:main} are local, valid in a \nbhd of a
point on the target manifold $\Gamma^\star$. In this section we seek a
global solution. By global we mean a solution valid in a \nbhd of
$\Gamma^\star$, not necessarily all of $\Real^n$. The following is a
global version of LTFLPI.

\begin{problem*}{Global Transverse Feedback Linearization with Partial
    Information (GTFLPI) Problem }
  Given a smooth single-input system~\eqref{eq:system} with smooth
  output~\eqref{eq:output}, a closed, connected, embedded,
  $n^\star$-dimensional controlled invariant submanifold $\Gamma^\star
  \subset \Real^n$, find, if possible, a diffeomorphism $\Xi \in
  \Diff{\mathcal{N}}$
\begin{equation}
\begin{aligned}
  \Xi  : \mathcal{N} &\to \Xi(\mathcal{N}) \subseteq \Gamma^\star \times
  \Real^{n -
    n^\star}\\
   x &\mapsto (\eta, \xi)
\end{aligned}
\end{equation}
where $\mathcal{N}$ is a \nbhd of $\Gamma^\star$, such that
\begin{itemize}
\item[(i)] The restriction of $\Xi$ to $\Gamma^\star$ is
\[
\left.\Xi\right|_{\Gamma^\star} : x \mapsto (\eta, 0). 
\]
\item[(ii)] The dynamics of system~\eqref{eq:system} in $(\eta,
  \xi)$-coordinates is given by~\eqref{eq:partial_info} where the pair
  $(A,b)$ is in Brunovsk\'{y} normal form (one chain of integrators)
  and $a_2(\eta, \xi) \neq 0$ in $\Xi(\mathcal{N})$.
\item[(iii)] The first component of $\xi$, denoted $\xi_1$ is
  observable, i.e., there exists a function $\tilde{\lambda} :
  h(\mathcal{N}) \subseteq \Real^p \to \Real$ such that
\[
\xi_1(x) = \tilde{\lambda} \circ h(x).
\] 
\end{itemize}
\end{problem*}
In order to solve GTFLPI we restrict the class of allowable target
sets $\Gamma^\star$.
\begin{assumption}
  The set $\Gamma^\star$ is diffeomorphic to a generalized cylinder,
  i.e., $\Gamma^\star \simeq \mathbb{T}^k \times \Real^{n^\star - k}$, $k
  \in \{0, \ldots, n^\star\}$ where $\mathbb{T}^k$ is the $k$-torus.
\label{ass:cylinder}
\end{assumption}

Assumption~\ref{ass:cylinder} restricts the class of set considered in
the global problem. However, there are many applications, most notably
path following, where the set to be stabilized is a generalized
cylinder~\cite{ConMagNieTos10}. This type of set arises when solving a
path following problem for closed curves.



\begin{proposition}
\label{prop:global_normal_form}
Under Assumption~\ref{ass:cylinder}, GTFLPI is solvable if
\begin{itemize}
\item[(a)] $(\forall x \in \Gamma^\star)$ $T_{x}\Gamma \oplus \Gg_{n-n^\star-1}(x) = T_{x}\Real^{{n}}$
\item[(b)] The distribution $\Gg_{n - n^\star -2}$ is non-singular and
  involutive in a \nbhd of $\Gamma^\star$
\item[(c)] The distribution $\inv{\left(\Gg_{{n}-n^\star -2}+
  \W\right)}$ is non-singular and, $(\forall x \in \Gamma^\star)$, 
\[
\dim{\left(T_x\gstar \oplus \Gg_{{n}-n^\star -2}(x)\right)} =
\dim{\left(T_x\gstar \oplus \inv{\left(\Gg_{{n}-n^\star -2}+
        \W\right)}(x)\right)}.
\]
\end{itemize}
\end{proposition}

\begin{proof}
  The proof of this result is based on the proof of~\cite[Theorem
  4.4]{NieMag06}. Assumption~\ref{ass:cylinder} and hypotheses (a) and
  (b) imply that the conditions of~\cite[Theorem 4.4]{NieMag06}
  hold. This guarantees the existence of a diffeomorphism with
  properties (i) and (ii) in the GTFLPI problem. In particular it
  ensures the existence of a function $\lambda : \mathcal{N} \subseteq
  \Real^{{n}} \To \Real$ such that $\Gamma^\star \subset
  \lambda^{-1}(0) = \set{x \in \mathcal{N} : \lambda(x) = 0}$ and
  which yields a uniform relative degree of ${n} - n^\star$ over
  $\Gamma^\star$.

  Now, assume that (c) holds, then using similar arguments to those in
  Theorem~\ref{thm:main} we have
\[
\begin{aligned}
\D\lambda \in \ann{(T\Gamma^\star)} \cap \ann{({\Gg}_{n - n^\star
-2})} &= \ann{(T\Gamma^\star + {\Gg}_{n - n^\star -2})}\\ &=
\ann{(T\Gamma^\star + \inv{\left(\Gg_{{n} - n^\star - 2} + W\right)})}
\\&\subseteq \ann{(T\Gamma^\star + \Gg_{{n} - n^\star - 2} + W)}\\
&= \ann{(T\Gamma^\star)}\cap \ann{(\Gg_{{n} - n^\star - 2} + W)}.
\end{aligned}
\]
Therefore,
\[
\begin{aligned}
\D\lambda &\in \ann{(\Gg_{{n} - n^\star - 2})} \cap \ann{(W)}\\&=
\ann{(\Gg_{{n} - n^\star - 2})} \cap \Sp\{\D h_1, \ldots,
\D h_p\}
\end{aligned}
\]
so
\[
\D\lambda = \sum^{p}_{i=1}\sigma_i(x)\D h_i(x)
\]
which implies that $\lambda = \tilde{\lambda}(h(x))$ with $\sigma_i(x)
= \left.\DER{\tilde{\lambda}}{y_i}\right|_{y = h(x)}$ and therefore
that $\xi_1(x) = \lambda(x)$ is observable.
\end{proof}

The next result is partly motivated by the results
in~\cite{MarPraIsi07} as discussed in
Section~\ref{sec:problem}. Recall that once a system is represented in
the normal form~\eqref{eq:partial_info}, the approach one uses to
stabilize $\Gamma^\star$ using output feedback depends on the dynamics
of the tangential subsystem~\eqref{eq:tang}. When the tangential
dynamics do not necessarily converge to zero but remain bounded, one
may use the results in~\cite{MarPraIsi07}. However,
in~\cite{MarPraIsi07} the $\eta$-dynamics in~\eqref{eq:partial_info}
must only depend on the tangential states $\eta$ and the observable
transversal state $\xi_1$. The next result gives sufficient conditions
for this to be the case.

\begin{corollary}
  If, in addition to the hypothesis of
  Proposition~\ref{prop:global_normal_form},
\begin{itemize}
\item[(d)] For all $i, j \in \set{0, \ldots, n - n^\star - 1}$,
  $\liebr{ad^i_fg}{ad^j_fg} = 0$
\end{itemize}
then there exists a diffeomorphism solving the GTFLPI problem such
that the dynamics of system~\eqref{eq:system} in $(\eta,
\xi)$-coordinates is given by
\begin{equation}
\begin{aligned}
\dot{\eta} &= f_0(\eta, \xi_1)\\
\dot{\xi} &= A\xi + b(a_1(\eta, \xi) + a_2(\eta, \xi)u),
\end{aligned}
\label{eq:partial_info_special}
\end{equation}
with $(A, b)$ a controllable pair in Brunovsk\'{y} normal
form. 
\label{cor:simple_tang}
\end{corollary}

\begin{proof}
  By Proposition~\ref{prop:global_normal_form} there exists a
  real-valued function $\lambda : \mathcal{N} \subseteq \Real^{{n}} \To
  \Real$ with $\Gamma^\star \subset \lambda^{-1}(0)$ and yields a
  uniform relative degree of ${n} - n^\star$ over $\Gamma^\star$. The
  remainder of the claim follows from the proof of~\cite[Proposition
  9.1.1]{Isi95} and the assumption that (d) holds. In this case, since
  the vector fields $ad^i_fg$, $i \in \set{1, \ldots, n - n^\star
    - 1}$ are not assumed to be complete, we are only assured to
  obtain a diffeomorphism with the required properties in a \nbhd of
  $\Gamma^\star$.
\end{proof}

\subsection{Systems with unobservable disturbances}
\label{sec:disturbances}

Theorem~\ref{thm:main} finds applications on systems affected by
disturbance signals that cannot be measured. Consider the single-input
smooth control system
\begin{equation}
 \label{eq:SISOexo}
\begin{aligned}
 \dot{x} &= f(x, w) + g(x, w)u\\
 \dot{w} &= s(w),
\end{aligned}
\end{equation}
with $x \in \Real^n$, $w \in \Real^k$ and control input $u \in
\Real$. The state $w$ can be thought of as an unobservable disturbance
generated by the dynamical system modeled by $\dot{w} = s(w)$. The
variable $x$ is assumed measured. Denote $\bar{n} \coloneqq n + k$ and
let $q \coloneqq \col{(x, w)}$. Re-writing~\eqref{eq:SISOexo} as
\begin{equation}
\dot{q} = F(q) + G(q)u
\label{eq:q_model}
\end{equation}
with $F(q) \coloneqq \col{(f(x,w), s(w))}$, $G(q) \coloneqq
\col{(g(x,w), 0)}$ and with output
\[
y = H(q), \qquad H(q) = \left[\begin{array}{cc}I_{n \times n} & 0_{n \times
      k}\end{array}\right]q,
\]
we immediately have the following corollary.
\begin{corollary}
\label{cor:disturbances}
Suppose that $\Gamma^\star \subset \Real^{\bar{n}}$ is a closed,
connected, embedded, $n^\star$-dimensional controlled invariant
submanifold for~\eqref{eq:SISOexo}. Let $q_0 = (x_0, w_0) \in
\Gamma^\star$ and
\[
\W \coloneqq \ann{\left(\Sp{\left\{\D H_1, \ldots , \D H_n\right\}}\right)}
= \Sp{\left\{\DER{}{w_1}, \ldots , \DER{}{w_{k}}\right\}}.
\]
If the hypotheses of Theorem~\ref{thm:main} hold, then there exists a
diffeomorphism
\[
\begin{aligned}
\Xi : &U \To \Xi(U)\\
&q \mapsto (\eta, \xi)
\end{aligned}
\]
where $U$ is a \nbhd of $q_0$ in $\Real^{\bar{n}}$, such that
system~\eqref{eq:SISOexo} is locally diffeomorphic
to~\eqref{eq:partial_info}. Moreover, the set $\Gamma^\star \cap U$ in
$(\eta, \xi)$-coordinates is given by
\[
\Xi(\Gamma^\star \cap U) = \set{(\eta, \xi) \in \Xi(U): \xi = 0}.
\]
Finally, the first component of $\xi$, denoted $\xi_1$ is observable.
\end{corollary}

As an application of these ideas, consider the path following problem
in presence of disturbances. Path following problems can naturally be
cast as set stabilization problems~\cite{NieMag06} and transverse
feedback linearization has been effectively used to implement path
following controllers~\cite{NieFulMag10}. When the control system is
affected by unobservable disturbances, the results in this paper are
more suitable to solving the path following problem.

\subsection*{Example}
  Consider a kinematic unicycle with unit translational velocity
  affected by an unobservable disturbance $w$ described by the model
\begin{equation}
\left[\begin{array}{c}\dot{x}_1 \\ \dot{x}_2 \\
    \dot{x}_3\end{array}\right] =
    \left[\begin{array}{c}\cos{(x_3)}\\ \sin{(x_3)}\\
    0\end{array}\right] +
    \left[\begin{array}{c}0\\0\\1\end{array}\right]u +
    \left[\begin{array}{cc}0 & 0\\ 1 & 0 \\ 0 & 0\end{array}\right]w.
\label{eq:unicycle}
\end{equation}
Suppose that the disturbance $w$ is generated by the exosystem
\begin{equation}
\left[\begin{array}{c}\dot{w}_1\\\dot{w}_2\end{array}\right] =
\left[\begin{array}{r}w_2\\-w_1\end{array}\right].
\label{eq:disturbance}
\end{equation}

We assume that the disturbance satisfies, for all $t \in \Real$,
$\|w(t)\| < 1$. Our objective is to design a control law that makes
the position $(x_1, x_2)$ of the unicycle converge to the unit circle
in the $(x_1, x_2)$-plane. The measured states are $x_1$, $x_2$,
$x_3$.

In the state space of~\eqref{eq:unicycle},~\eqref{eq:disturbance} we
view the goal set as the maximal controlled invariant set contained in
$\{(x,w) \in \Real^3 \times \Real^2: x_1^2 + x_2^2 - 1 = 0\}$ which is
given by
\begin{equation}
\Gamma^\star = \{(x, w) \in \Real^{3} \times \Real^2: x_1^2 + x_2^2 - 1 =
x_1\cos{(x_3)} + x_2\sin{(x_3)} + x_2w_1= 0\}.
\label{eq:gstar_ex}
\end{equation}
Here, $n^\star = \dim{(\gstar)} = 3$. Writing $\Gamma^\star =
\gamma^{-1}(0)$ for the function $\gamma : \Real^3 \times \Real^2 \to
\Real^2$ defined by~\eqref{eq:gstar_ex}, one can verify that, under the
assumption that $\|w\| < 1$, zero is a regular value of $\gamma$. The
set~\eqref{eq:gstar_ex} has two connected components, one
corresponding to clockwise motion along the path, the other
corresponding to counterclockwise motion. Since our results are local,
a choice of $(x_0, w_0) \in \Gamma^\star$ selects one of these two
components.

To apply Corollary~\ref{cor:disturbances}, set $q \coloneqq \col{(x,
  w)}$, and define $F(q) = \col{\left(\cos{(x_3)}\right.}$,
$\left.\sin{(x_3)} + w_1,0,w_2,-w_1\right)$, $G(q) =
\col{\left(0,0,1,0,0\right)}$, $H(q) = \col{\left(x_1, x_2,
    x_3\right)}$. The question we ask is: can transverse feedback
linearization be used to stabilize $\Gamma^\star$ using only $y$ for
feedback?  The answer, of course, is yes. For, one can check that the
observable function $\lambda(q) = x_1^2 + x_2^2 - 1$ yields a
well-defined relative degree of $2 = n - \ns$ at each $q \in
\Gamma^\star$. If we set $\xi_1(q) = \lambda(q)$, $\xi_2(q) =
L_F\lambda(q)$ and let $\eta_1, \eta_2, \eta_3$ be any\footnote{If the
  functions $\eta_1$, $\eta_2$, $\eta_3$ are chosen so that $\D\eta_1,
  \D\eta_2, \D\eta_3 \in \ann{(G(q))}$, then we obtain the normal
  form~\eqref{eq:partial_info}. It is always possible to do this. If,
  on the other hand, $\D\eta_1, \D\eta_2, \D\eta_3 \not \in
  \ann{(G(q))}$, then the control $u$ will appear in the $\eta$
  subsystem.} three additional linearly independent functions, then,
the local diffeomorphism $\Xi(q) = \col{(\eta(q), \xi(q))}$ solves
LTFLPI.

Next we confirm this observation using
Corollary~\ref{cor:disturbances}. In this case, with $q_0 = (x_0,
w_0)$, conditions $(a)$ and $(b)$ of Theorem~\ref{thm:main} become
\begin{itemize}
\item[$(a)$] $\dim{(T_{q_0}\Gamma^\star + \Gg_{1}(q_0))} = 5$

\item[$(b)$] there exists an open \nbhd $U$ of $q_0$ in
$\Real^5$ such that
\[
(\forall \; q \in \gstar\cap U) \; \dim(T_{q}\Gamma^\star +
\Gg_{0}(q)) = \dim(T_{q}\Gamma^\star + \inv{\left({\Gg}_{0} + \W\right)}(q))=
\text{ constant,}
\]
\end{itemize}
where $\W$, in this case, is given by $\W = \ann{\left(\Sp\{\D H_1, \D
    H_2, \D H_3\}\right)} = \Sp\{e_4, e_5\}$
where $e_4$ and $e_5$ are the fourth and fifth natural basis vectors
for $\Real^5$. At each $q \in \Gamma^\star$ we have that
\[
T_q\Gamma^\star =
\Sp{\left\{\left[\begin{array}{c}0\\0\\0\\0\\1\end{array}\right],
    \left[\begin{array}{r}x_2^2\\-x_1x_2\\0\\\phi(q)\\0\end{array}\right],
    \left[\begin{array}{r}-x_2\varphi(q)\\x_1\varphi(q)\\\phi(q)\\0\\0\end{array}\right]
  \right\}}
\]
where $\phi(q) \coloneqq x_1\left(\sin{(x_3)} + w_1\right) -
x_2\cos{(x_3)}$, $\varphi(q) \coloneqq  x_1\sin{(x_3)} - x_2\cos{(x_3)}$
are both non-zero on $\Gamma^\star$ if $\|w\| < 1$. Therefore
\[
T_q\Gamma^\star + \Gg_1(q)=
\Sp{\left\{\left[\begin{array}{c}0\\0\\0\\0\\1\end{array}\right],
    \left[\begin{array}{r}x_2^2\\-x_1x_2\\0\\\phi(q)\\0\end{array}\right],
    \left[\begin{array}{r}-x_2\varphi(q)\\x_1\varphi(q)\\\phi(q)\\0\\0\end{array}\right],
    \left[\begin{array}{c}0\\0\\1\\0\\0\end{array}\right],
    \left[\begin{array}{c}-\sin{(x_3)}\\\cos{(x_3)}\\0\\0\\0\end{array}\right]
  \right\}}.
\]
The determinant of the matrix whose columns are the above basis
vectors is given by $\phi(q)\varphi^2(q)$.
If this determinant is non-zero then condition (a) holds. To this end,
note that on the set~\eqref{eq:gstar_ex}
\[
\begin{aligned}
&x_1\cos{(x_3)} + x_2\sin{(x_3)} + x_2w_1= 0\\
\Rightarrow & x_2\left(\sin{(x_3)} + w_1\right) =
-x_1\cos{(x_3)}\\
\Rightarrow & x^2_2\left(\sin{(x_3)} + w_1\right)^2 =
x_1^2\cos^2{(x_3)}\\
\Rightarrow & x^2_2\left(\sin{(x_3)} + w_1\right)^2 =
(1 - x_2^2)\cos^2{(x_3)} \qquad (\text{since $x_1^2 + x_2^2 - 1 =
  0$})\\
\Rightarrow & x_2 = \pm \frac{\cos{(x_3)}}{\sqrt{\left(\sin{(x_3)} +
      w_1\right)^2 + \cos^2{(x_3)}}}.
\end{aligned}
\]
Similarly, using $x_1^2 = 1 - x_2^2$,
\[
x_1 = \mp \frac{\sin{(x_3)} + w_1}{\sqrt{\left(\sin{(x_3)} +
      w_1\right)^2 + \cos^2{(x_3)}}}.
\]
Therefore, we have that
\[
\phi(q) = \mp \frac{1 + 2w_1\sin{(x_3)} + w_1^2}{\sqrt{\left(\sin{(x_3)} + w_1\right)^2 +
  \cos^2{(x_3)}}}, \qquad \varphi(q) = \mp \frac{1 +
  w_1\sin{(x_3)}}{\sqrt{\left(\sin{(x_3)} + w_1\right)^2 + \cos^2{(x_3)}}}
\]
which, under the assumption that $\|w\| < 1$, are both non-zero.

To check condition (b) we first note that since $\Gg_0$ and $W$ are
constant distributions, $\Gg_0 + \W = \inv{\left(\Gg_0 +
    \W\right)}$. We have that, for any $q \in \Gamma^\star$,
\[
T_q\Gamma^\star + \Gg_0(q)=
\Sp{\left\{\left[\begin{array}{c}0\\0\\0\\0\\1\end{array}\right],
    \left[\begin{array}{r}x_2^2\\-x_1x_2\\0\\\phi(q)\\0\end{array}\right],
    \left[\begin{array}{r}-x_2\varphi(q)\\x_1\varphi(q)\\\phi(q)\\0\\0\end{array}\right],
    \left[\begin{array}{c}0\\0\\1\\0\\0\end{array}\right]\right\}},
\]
and
\[
T_q\Gamma^\star + \inv{(\Gg_0 + \W)}(q)=
\Sp{\left\{\left[\begin{array}{c}0\\0\\0\\0\\1\end{array}\right],
    \left[\begin{array}{r}x_2^2\\-x_1x_2\\0\\\phi(q)\\0\end{array}\right],
    \left[\begin{array}{r}-x_2\varphi(q)\\x_1\varphi(q)\\\phi(q)\\0\\0\end{array}\right],
    \left[\begin{array}{c}0\\0\\1\\0\\0\end{array}\right],
    \left[\begin{array}{c}0\\0\\0\\1\\0\end{array}\right]\right\}}.
\]
Since, as already shown, when $\|w\| < 1$ the functions $\varphi$ and
$\phi$ are non-zero on $\Gamma^\star$ and since $x_1$ and $x_2$ are
not equal to zero simultaneously, condition (b) is satisfied.

In $(\eta, \xi)$-coordinates the unicycle has the normal
form~\eqref{eq:partial_info}. At this point a ``high-gain'' output
feedback controller can be used to make the target set $\Gamma^\star$
attractive and have the unicycle traverse the desired path in the
presence of unobservable disturbances.

\section{Conclusions}

In this paper we studied the problem of stabilizing a controlled
invariant embedded submanifold in the state space of autonomous
nonlinear control systems using output feedback. We studied the most
natural approach to solving this problem : given a controlled
invariant manifold (the target set), find an observable output
function yielding a well-defined relative degree whose associated zero
dynamics manifold locally coincides with the target set. We call this
the local transverse feedback linearization problem with partial
information. Necessary and sufficient conditions were presented under
which this problem is solvable. We also presented a global solution to
this problem in the case when the target set is a generalized
cylinder.  Finally we illustrated how this work may find applications
in system affected by unobservable disturbances and to path following
problems.


%
%
\bibliographystyle{siam}
\bibliography{bibTFLPI}

\end{document}